\documentclass[11pt]{article}
\usepackage[margin=1in]{geometry}

\usepackage{graphicx}
\usepackage{amsmath}
\usepackage{amssymb}
\usepackage{amsthm}

\usepackage{color}
\allowdisplaybreaks[1]

\usepackage{algorithm}
\usepackage{algpseudocode}
\usepackage{hyperref}

\theoremstyle{plain}
\newtheorem{theorem}{Theorem}[section]

\newtheorem{cor}[theorem]{Corollary}

\theoremstyle{definition}

\theoremstyle{remark}

\newtheorem{remark}{Remark}[section]

\def\0{{\bf 0}}
\def\1{{\bf 1}}

\def \bmat{\left[\begin{matrix}}
	\def \emat{\end{matrix}\right]}
\def \bvec{\left(\begin{matrix}}
	\def \evec{\end{matrix}\right)}

\def \xy1vec{\left[\begin{matrix}x\\y\\1\end{matrix}\right]}
\def \QED{\begin{flushright}\Halmos\end{flushright}\end{proof}}
\def \defeq{\mathrel{\mathop{:}}=}

\def \xbar{\bar{x}}

\def \grad{\triangledown}

\def \gp3d{\grad p_3(d)}
\def \Hess{\triangledown^2}
\def \R{\mathbb{R}}
\def \Rn{\R^n}
\def \beq{\begin{equation}}
\def \eeq{\end{equation}}
\def \baeq{\begin{equation*}\begin{aligned}}
	\def \eaeq{\end{aligned}\end{equation*}}
\newcommand{\baeql}[1]{\begin{equation}\label{#1}\begin{aligned}}
	\def \eaeql{\end{aligned}\end{equation}}
\def \otm{\{1, \ldots, m\}}

\newcommand{\ceil}[1]{\lceil #1 \rceil}

\title{\LARGE \bf On the Complexity of Finding a Local Minimizer\\ of a Quadratic Function over a Polytope}
\author{Amir Ali Ahmadi\thanks{Amir Ali Ahmadi is with the department of Operations Research and Financial Engineering at Princeton University. Email: \texttt{aaa@princeton.edu}.} and Jeffrey Zhang\thanks{Jeffrey Zhang is with the department of Mathematical Sciences at Carnegie Mellon University. Email: \texttt{jeffz@cmu.edu}.} \thanks{This work was partially supported by an AFOSR MURI award, the DARPA Young Faculty Award, the Princeton SEAS Innovation Award, the NSF CAREER Award, the Google Faculty Award, and the Sloan Fellowship.}}
\begin{document}
\date{}
\maketitle
\begin{abstract}
	\noindent
	We show that unless P=NP, there cannot be a polynomial-time algorithm that finds a point within Euclidean distance $c^n$ (for any constant $c \ge 0$) of a local minimizer of an $n$-variate quadratic function over a polytope. This result (even with $c=0$) answers a question of Pardalos and Vavasis that appeared in 1992 on a list of seven open problems in complexity theory for numerical optimization. Our proof technique also implies that the problem of deciding whether a quadratic function has a local minimizer over an (unbounded) polyhedron, and that of deciding if a quartic polynomial has a local minimizer are NP-hard.
\end{abstract}


\section{Introduction}\label{Sec: Introduction}

Recall that a \emph{local minimizer} of a function $f: \Rn \to \R$ over a set $\Omega \subseteq \Rn$ is a point $\xbar \in \Omega$ for which there exists a scalar $\epsilon > 0$ such that $f(\xbar) \le f(x)$ for all $x \in \Omega$ with $\|x - \xbar\| \le \epsilon$. In the case where $f$ is a linear function and $\Omega$ is polyhedral (i.e., the case of linear programming), it is well known that a local minimizer (which also has to be a global minimizer) can be found in polynomial time in the Turing model of computation~\cite{khachiyan1979polynomial,karmarkar1984new}. Perhaps the next simplest constrained optimization problem to consider is one where $f$ is a quadratic function and $\Omega$ is polyhedral. Such an optimization problem is known as a \emph{quadratic program} and can be written as
\beq\label{Eq: Polynomial Optimization Problem}
\begin{aligned}
	& \underset{x \in \Rn}{\min}
	& & x^TQx + c^Tx \\
	& \text{subject to}
	&& a_i^Tx \le b_i, \forall i \in \otm,\\
\end{aligned}
\eeq
where $Q \in \R^{n \times n}$, $c, a_1, \ldots, a_m \in \Rn$, and $b_1, \ldots, b_m~\in~\R$. The matrix $Q$ is taken without loss of generality to be symmetric. When complexity questions about quadratic programs are studied in the Turing model of computation, all these data are rational and the input size is the total number of bits required to write them down. When the matrix $Q$ is positive semidefinite, then we recover the case of convex quadratic programming, where it is well known that finding a global minimizer can be done in polynomial time~\cite{kozlov1980polynomial}. However, when $Q$ has even a single negative eigenvalue, finding a global minimizer is NP-hard~\cite{pardalos1991quadratic}. It is therefore natural to ask whether one can instead find a local minimizer of a general quadratic program efficiently. In fact, this precise question appeared in 1992 on a list of seven open problems in complexity theory for numerical optimization \cite{pardalos1992open}:
\begin{quote}
	\emph{``What is the complexity of finding even a local minimizer for nonconvex quadratic programming, assuming the feasible set is compact? Murty and Kabadi (1987, \cite{murty1987some}) and Pardalos and Schnitger (1988, \cite{pardalos1988checking}) have shown that it is NP-hard to test whether a given point for such a problem is a local minimizer, but that does not rule out the possibility that another point can be found that is easily verified as a local minimizer.''}
\end{quote}

A few remarks on the phrasing of this problem are in order. First, note that in this question, the feasible set of the quadratic program is assumed to be compact (i.e., a polytope). Therefore, there is no need to focus on the related and often prerequisite problem of deciding the \emph{existence} of a local minimizer (since any global minimizer e.g. is a local minimizer). The latter question makes sense in the case where the feasible set of the quadratic program is unbounded; the complexity of this question is also addressed in our paper (Theorem~\ref{Thm: LM QP NP hard}). Second, as the quote points out, the question of finding a local minimizer is also separate from a complexity viewpoint from that of testing if a given point is a local minimizer. This related question has been studied more extensively and its complexity has already been settled for optimization problems whose objective and constraints are given by polynomial functions of any degree; see~\cite{murty1987some,pardalos1988checking,cubicpaper}.

To point out some of the subtle differences between these variations of the problem more specifically, we briefly review the reduction of Murty and Kabadi \cite{murty1987some}, which shows the NP-hardness of deciding if a given point is a local minimizer of a quadratic program. In~\cite{murty1987some}, the authors show that the problem of deciding if a symmetric matrix $Q$ is \emph{copositive}---i.e. whether $x^TQx \ge 0$ for all vectors $x$ in the nonnegative orthant---is NP-hard. From this, it is straightforward to observe that the problem of testing whether a given point is a local minimizer of a quadratic function over a polyhedron is NP-hard: Indeed, the origin is a local minimizer of $x^TQx$ over the nonnegative orthant if and only if the matrix $Q$ is copositive. However, it is not true that $x^TQx$ has a local minimizer over the nonnegative orthant if and only if $Q$ is copositive. Although the ``if'' direction holds, the ``only if'' direction does not. For example, consider the matrix
$$Q = \bmat 0 & 1 \\ 1 & -2\emat,$$ which is clearly not copositive, even though the point $(1,0)^T$ is a local minimizer of $x^TQx$ over the nonnegative orthant.

Our main results in this paper are as follows. We show that unless P=NP, no polynomial-time algorithm can find a point within Euclidean distance $c^n$ (for any constant $c \ge 0$) of a local minimizer of an $n$-variate quadratic program with a compact feasible set (Theorem~\ref{Thm: Finding LM NP-hard}). See also Corollaries~\ref{Cor: QP Pseudo} and \ref{Cor: 2n NP-hard}. To prove this, we show as an intermediate step that deciding whether a quartic polynomial or a quadratic program has a local minimizer is strongly NP-hard\footnote{This implies that these
	problems remain NP-hard even if the bitsize of all numerical data are $O(\log n)$, where $n$ is the number of variables. For a
	strongly NP-hard problem, even a pseudo-polynomial time algorithm---i.e., an algorithm whose running time is polynomial in the magnitude of the numerical data of the problem but not necessarily in their bitsize---cannot exist unless P=NP.  See \cite{garey2002computers} or \cite[Section 2]{ahmadi2019complexity} for more details.} (Theorems~\ref{Thm: LM Degree 4 NP hard} and \ref{Thm: LM QP NP hard}). Finally, we show that unless P=NP, there cannot be a polynomial-time algorithm that decides if a quadratic program with a compact feasible set has a unique local minimizer and if so returns this minimizer (Theorem~\ref{Thm: UniqueQP NP Hard}).

Overall, our results suggest that without additional problem structure, questions related to finding local minimizers of quadratic programs are not easier (at least from a complexity viewpoint) than those related to finding global minimizers. It also suggests that any efficient heuristic that aims to find a local minimizer of a quadratic program must necessarily fail on a ``significant portion'' of instances; see e.g. Corollary 2.2 of \cite{hemaspaandra2012sigact} for a more formal complexity theoretic statement.

\subsection{Notation and Basic Definitions}\label{SSec: Notation}

For a vector $x \in \Rn$, the notation $x^2$ denotes the vector in $\Rn$ whose $i$-th entry is $x_i^2$, and $Diag(x)$ denotes the diagonal $n \times n$ matrix whose $i$-th diagonal entry is $x_i$. The notation $x \ge 0$ denotes that the vector $x$ belongs to the nonnegative orthant, and for such a vector, $\sqrt{x}$ denotes the vector in $\Rn$ whose $i$-th entry is $\sqrt{x_i}$. For two matrices $X, Y \in \R^{m \times n}$, we denote by $X \odot Y$ the matrix in $\R^{m \times n}$ whose $(i,j)$-th entry is $X_{ij}Y_{ij}$. For vectors $x, y \in \Rn$, the notation $y_x$ (sometimes $(y)_x$ if there is room for confusion with other indices) denotes the vector containing the entries of $y$ where $x_i$ is nonzero in the same order (the length of $y_x$ is hence equal to the number of nonzero entries in $x$). Similarly, for a vector $x \in \Rn$ and a matrix $Y \in \R^{n \times n}$, the notation $Y_x$ (sometimes $(Y)_x$ if there is room for confusion with other indices) denotes the principal submatrix of $Y$ consisting of rows and columns of $Y$ whose indices correspond to indices of nonzero entries of $x$. The notation $I$ (resp. $J$) refers to the identity matrix (resp. the matrix of all ones); the dimension will be clear from context. For a symmetric matrix $M \in \R^{n \times n}$, the notation $M \succeq 0$ (resp. $M \succ 0$) denotes that $M$ is \emph{positive semidefinite} (resp. \emph{positive definite}), i.e. that it has nonnegative (resp. positive) eigenvalues. As mentioned already, we say that $M$ is \emph{copositive} if $x^TMx \ge 0, \forall x \ge 0$. The \emph{simplex} in $\Rn$ is denoted by $\Delta_n \defeq \{x \in \Rn\ |\ x \ge 0, \sum_{i=1}^n x_i = 1\}$. Finally, for a scalar $c$, the notation $\ceil{c}$ denotes the ceiling of $c$, i.e. the smallest integer greater than or equal to $c$.

We recall that a \emph{form} is a homogeneous polynomial; i.e. a polynomial whose monomials all have the same degree. A form $p: \Rn \to \R$ is said to be \emph{nonnegative} if $p(x) \ge 0, \forall x \in \Rn$, and \emph{positive definite} if $p(x) > 0, \forall x \ne 0$. A \emph{critical point} of a differentiable function $f:\Rn \to \R$ is a point $x\in\Rn$ at which the gradient $\grad  f(x)$ is zero. A \emph{second-order point} of a twice-differentiable function $f:\Rn \to \R$ is a critical point $x$ at which the Hessian matrix $\Hess f(x)$ is positive semidefinite.

All graphs in this paper are undirected, unweighted, and have no self-loops. The \emph{adjacency matrix} of a graph $G$ on $n$ vertices is the $n \times n$ symmetric matrix whose $(i,j)$-th entry equals one if vertices $i$ and $j$ share an edge in $G$ and zero otherwise. The \emph{complement} of a graph $G$, denoted by $\bar{G}$, is the graph with the same vertex set as $G$ and such that two distinct vertices are adjacent if and only if they are not adjacent in $G$. An \emph{induced subgraph} of $G$ is a graph containing a subset of the vertices of $G$ and all edges connecting pairs of vertices in that subset.

\section{The Main Result}\label{Sec: Main Result}

\subsection{Complexity of Deciding Existence of Local Minimizers}

To show that a polynomial-time algorithm for finding a local minimizer of a quadratic function over a polytope (i.e., a bounded polyhedron) implies P = NP (Theorem~\ref{Thm: Finding LM NP-hard}), we first show that it is NP-hard to decide whether a quartic polynomial has a local minimizer (Theorem~\ref{Thm: LM Degree 4 NP hard}). The main idea of this proof is used to show the same claim about deciding whether a quadratic program has a local minimizer (Theorem~\ref{Thm: LM QP NP hard}) and to then establish our main result (Theorem~\ref{Thm: Finding LM NP-hard}).

\begin{theorem}\label{Thm: LM Degree 4 NP hard}
	It is strongly NP-hard to decide if a degree-4 polynomial has a local minimizer.
\end{theorem}

We will prove this theorem by presenting a polynomial-time reduction from the STABLESET problem, which is known to be (strongly) NP-hard~\cite{garey2002computers}. Recall that in the STABLESET problem, we are given as input a graph $G$ on $n$ vertices and a positive integer $r \le n$. We are then asked to decide whether $G$ has a \emph{stable set} of size $r$, i.e. a set of $r$ pairwise non-adjacent vertices. We denote the size of the largest stable set in a graph $G$ by the standard notation $\alpha(G)$. We also recall that a \emph{clique} in a graph $G$ is a set of pairwise adjacent vertices. The size of the largest clique in $G$ is denoted by $\omega(G)$. The following theorem of Motzkin and Straus~\cite{motzkin1965maxima} relates $\omega(G)$ to the optimal value of a quadratic program.

\begin{theorem}[\cite{motzkin1965maxima}]\label{Thm: Motzkin Straus}
	Let $G$ be a graph on $n$ vertices with adjacency matrix $A$ and clique number $\omega$. The optimal value of the quadratic program
	
	\begin{equation}\label{Eq: MS QP}
		\begin{aligned}
			& \underset{x \in \Rn}{\max}
			& & x^TAx \\
			& \text{\emph{subject to}}
			&& x \ge 0,\\
			&&& \sum_{i=1}^n x_i = 1
		\end{aligned}
	\end{equation}
	is $1 - \frac{1}{\omega}$.
\end{theorem}

For a scalar $k$ and a symmetric matrix $A$ (which will always be an adjacency matrix), the following notation will be used repeatedly in our proofs:
\begin{equation}\label{Def: MAk} M_{A,k} \defeq kA + kI - J,\end{equation}
\begin{equation}\label{Def: qAk}q_{A,k}(x) \defeq x^TM_{A,k}x,\end{equation}
and
\begin{equation}\label{Def: pAk}p_{A,k}(x) \defeq (x^2)^T M_{A,k} x^2.\end{equation}
Note that nonnegativity of the quadratic form $q_{A,k}$ over the nonnegative orthant is equivalent to (global) nonnegativity of the quartic form $p_{A,k}$ and to copositivity of the matrix $M_{A,k}$.

The following corollary of Theorem~\ref{Thm: Motzkin Straus} will be of more direct relevance to our proofs. The first statement in the corollary has been observed e.g. by de Klerk and Pashechnik~\cite{de2002approximation}, but its proof is included here for completeness. The second statement, which will also be needed in the proof of Theorem~\ref{Thm: LM Degree 4 NP hard}, follows straightforwardly.

\begin{cor}\label{Cor: Motzkin Straus 2}
	For a scalar $k > 0$ and a graph $G$ with adjacency matrix $A$, the matrix $M_{A,k}$ in (\ref{Def: MAk}) is copositive if and only if $\alpha(G) \le k$. Furthermore, if $\alpha(G) < k$, the quartic form $p_{A,k}$ in (\ref{Def: pAk}) is positive definite.\footnote{The converse of this statement also holds, but we do not need it for the proof of Theorem~\ref{Thm: LM Degree 4 NP hard}.}
\end{cor}
\begin{proof}
	First observe that $\alpha(G) = \omega(\bar{G})$, and that the adjacency matrix of $\bar{G}$ is $J - A - I$. Thus from Theorem~\ref{Thm: Motzkin Straus}, the maximum value of $x^T(J-A-I)x$ over $\Delta_n$ is $1 - \frac{1}{\alpha(G)}$, and hence the minimum value of $x^T(A+I)x$ over $\Delta_n$ is $\frac{1}{\alpha(G)}$. Therefore, for any $k > 0$, $\alpha (G) \le k$ if and only if $x^T(A+I)x \ge \frac{1}{k}$ for all $x \in \Delta_n$, which holds if and only if $x^T(k(A+I)-J)x \ge 0$ for all $x \in \Delta_n.$ The first statement of the corollary then follows from the homogeneity of $x^T(k(A+I)-J)x$.
	
	To show that $p_{A,k}$ is positive definite when $\alpha(G) < k$, observe that
	$$kA+kI-J = (\alpha(G)(A+I)-J) + (k - \alpha(G))(A+I).$$
	Considering the two terms on the right separately, we observe that $(x^2)^T (\alpha(G)(A+I)-J) x^2$ (i.e., $p_{A,\alpha(G)})$ is nonnegative since $M_{A,\alpha(G)}$ is copositive, and that $(x^2)^T (k-\alpha(G))(A+I)x^2$ is positive definite. Therefore, their sum $(x^2)^T(kA+kI-J )x^2$ is positive definite.
\end{proof}

We now present the proof of Theorem~\ref{Thm: LM Degree 4 NP hard}. While the statement of the theorem is given for degree-4 polynomials, it is straightforward to extend the result to higher-degree polynomials. We note that degree four is the smallest degree for which deciding existence of local minimizers is intractable. For degree-3 polynomials, it turns out that this question can be answered by solving semidefinite programs of polynomial size~\cite{cubicpaper}.

\begin{proof}[Proof (of Theorem~\ref{Thm: LM Degree 4 NP hard})]
	We present a polynomial-time reduction from the STABLESET problem. Let a graph $G$ on $n$ vertices with adjacency matrix $A$ and a positive integer $r \le n$ be given. We show that $G$ has a stable set of size $r$ if and only if the quartic form $p_{A,r-0.5}$ defined in (\ref{Def: pAk}) has no local minimizer. This is a consequence of the following more general fact that we prove below: For a noninteger scalar $k$, the quartic form $p_{A,k}$ has no local minimizer if and only if $\alpha(G) \ge k$.
	
	We first observe that if $\alpha(G) < k$, then $p_{A,k}$ has a local minimizer. Indeed, recall from the second claim of Corollary~\ref{Cor: Motzkin Straus 2} that under this assumption, $p_{A,k}$ is positive definite. Since $p_{A,k}$ vanishes at the origin, it follows that the origin is a local minimizer. Suppose now that $\alpha(G) \ge k$. Since $k$ is noninteger, this implies that $\alpha(G) > k$. We show that in this case, $p_{A,k}$ has no local minimizer by showing that the origin must be the only second-order point of $p_{A,k}$. Since any local minimizer of a polynomial is a second-order point, only the origin can be a candidate local minimizer for $p_{A,k}$. However, by the first claim of Corollary~\ref{Cor: Motzkin Straus 2}, the matrix $M_{A,k}$ is not copositive and hence $p_{A,k}$ is not nonnegative. As $p_{A,k}$ is homogeneous, this implies that $p_{A,k}$ takes negative values arbitrarily close to the origin, ruling out the possibility of the origin being a local minimizer.
	
	To show that when $\alpha(G) > k$, the origin is the only second-order point of $p_{A,k}$, we compute the gradient and Hessian of $p_{A,k}$. We have
	$$\grad p_{A,k}(x) = 4x \odot M_{A,k} x^2,$$
	and
	$$\Hess p_{A,k}(x) = 8M_{A,k} \odot xx^T + 4Diag(M_{A,k} x^2).$$
	Suppose for the sake of contradiction that $p_{A,k}$ has a nonzero second-order point $\xbar$. Since $\xbar$ is a critical point, $\grad p_{A,k}(\xbar)=0$ and thus $(M_{A,k} \xbar^2)_{\xbar} = 0$. It then follows that
	$$(\Hess p_{A,k}(\xbar))_{\xbar} = 8(M_{A,k})_{\xbar} \odot \xbar_{\xbar}\xbar_{\xbar}^T.$$
	Because $\Hess p_{A,k}(\xbar) \succeq 0$ and thus all its principal submatrices are positive semidefinite, we have $8(M_{A,k})_{\xbar} \odot \xbar_{\xbar}\xbar_{\xbar}^T \succeq 0$. Since
	$$(M_{A,k})_{\xbar} \odot \xbar_{\xbar}\xbar_{\xbar}^T = Diag(\xbar_{\xbar}) (M_{A,k})_{\xbar} Diag(\xbar_{\xbar}),$$
	and since $Diag(\xbar_{\xbar})$ is an invertible matrix, it follows that $(M_{A,k})_{\xbar} \succeq 0$.

	We now consider the induced subgraph $G_{\xbar}$ of $G$ with vertices corresponding to the indices of the nonzero entries of $\xbar$. Note that the adjacency matrix of $G_{\xbar}$ is $A_{\xbar}$. Furthermore, observe that $M_{A_{\xbar},k} = (M_{A,k})_{\xbar}$, and therefore $M_{A_{\xbar},k}$ is positive semidefinite and thus copositive. We conclude from the first claim of Corollary~\ref{Cor: Motzkin Straus 2} that $\alpha(G_{\xbar}) \le k$. We now claim that
	$$M_{A_{\xbar},k}\xbar_{\xbar}^2 = (M_{A,k})_{\xbar}\xbar_{\xbar}^2 = (M_{A,k}\xbar^2)_{\xbar} = 0.$$
	The first equality follows from that $M_{A_{\xbar},k} = (M_{A,k})_{\xbar}$, the second from that the indices of the nonzero entries of $\xbar$ are the same as those of $\xbar^2$, and the third from that $\grad p_{A,k} (\xbar) = 0$. Hence, $p_{A_{\xbar},k}(\xbar_{\xbar}) = 0$. Since $\xbar_{\xbar}$ is nonzero, $p_{A_{\xbar},k}$ is not positive definite. By the second claim of Corollary~\ref{Cor: Motzkin Straus 2}, we must have $\alpha(G_{\xbar}) \ge k$. Therefore, $\alpha(G_{\xbar}) = k$. However, because $k$ was assumed to be noninteger, we have a contradiction.
\end{proof}

It turns out that the proof of Theorem~\ref{Thm: LM Degree 4 NP hard} also shows that it is NP-hard to decide if a quartic polynomial has a strict local minimizer. Recall that a \emph{strict local minimizer} of a function $f: \Rn \to \R$ over a set $\Omega \subseteq \Rn$ is a point $\xbar \in \Omega$ for which there exists a scalar $\epsilon > 0$ such that $p(\xbar) < p(x)$ for all $x \in \Omega\backslash \xbar$ with $\|x - \xbar\| \le \epsilon$.

\begin{cor}\label{Cor: SLM Degree 4 NP hard}
	It is strongly NP-hard to decide if a degree-4 polynomial has a strict local minimizer.
\end{cor}
\begin{proof}
	Observe from the proof of Theorem~\ref{Thm: LM Degree 4 NP hard} that for a graph $G$ and a noninteger scalar $k$, the quartic form $p_{A,k}$ has a local minimizer if and only if $\alpha(G)<k$. In the case where $p_{A,k}$ does have a local minimizer, we showed that $p_{A,k}$ is positive definite, and thus the local minimizer (the origin) must be a strict local minimizer.
\end{proof}

We now turn our attention to local minimizers of quadratic programs.

\begin{theorem}\label{Thm: LM QP NP hard}
	It is strongly NP-hard to decide if a quadratic function has a local minimizer over a polyhedron. The same is true for deciding if a quadratic function has a strict local minimizer over a polyhedron.
\end{theorem}

\begin{proof}
	We present a polynomial-time reduction from the STABLESET problem to the problem of deciding if a quadratic function has a local minimizer over a polyhedron. The reader can check that same reduction is valid for the case of strict local minimizers.
	
	Let a graph $G$ on $n$ vertices with adjacency matrix $A$ and a positive integer $r \le n$ be given. Let $k = r - 0.5$, $q_{A,k}$ be the quadratic form defined in (\ref{Def: qAk}), and consider the optimization problem
	\begin{equation}\label{Eq: SS QP}
		\begin{aligned}
			& \underset{x \in \Rn}{\min}
			& & q_{A,k}(x) \\
			& \text{subject to}
			&& x \ge 0.
		\end{aligned}
	\end{equation}
	We show that a point $x \in \Rn$ is a local minimizer of (\ref{Eq: SS QP}) if and only if $\sqrt{x}$ is a local minimizer of the quartic form $p_{A,k}$ defined in (\ref{Def: pAk}). By the arguments in the proof of Theorem~\ref{Thm: LM Degree 4 NP hard}, we would have that (\ref{Eq: SS QP}) has no local minimizer if and only if $G$ has a stable set of size $r$.
	
	Indeed, if $x$ is not a local minimizer of (\ref{Eq: SS QP}), there exists a sequence $\{y_j\} \subseteq \Rn$ with $y_j \to x$ and such that for all $j$, $y_j \ge 0$ and $q_{A,k}(y_j) < q_{A,k}(x)$. The sequence $\{\sqrt{y_j}\}$ would then satisfy $p_{A,k}(\sqrt{y_j}) < p_{A,k}(\sqrt{x})$ and $\sqrt{y_j} \to \sqrt{x}$, proving that $\sqrt{x}$ is not a local minimizer of $p_{A,k}$. Similarly, if $x$ is not a local minimizer of $p_{A,k}$, there exists a sequence $\{z_j\} \subseteq \Rn$ such that $z_j \to x$ and $p_{A,k}(z_j) < p_{A,k}(x)$ for all $j$. The sequence $\{z_j^2\}$ would then prove that $x^2$ is not a local minimizer of (\ref{Eq: SS QP}).
	
\end{proof}

\begin{remark}
	The decision problems that Theorems~\ref{Thm: LM Degree 4 NP hard} and~\ref{Thm: LM QP NP hard} prove to be NP-hard are unlikely to be in NP. Indeed, membership of these problems in NP would imply (via our reductions) that STABLESET is in co-NP, which would further imply that NP=co-NP.
\end{remark}

\subsection{Complexity of Finding a Local Minimizer of a Quadratic Function Over a Polytope}

We now address the original question posed by Pardalos and Vavasis concerning the complexity of finding a local minimizer of a quadratic program with a compact feasible set. Note again that if the feasible set is compact, the existence of a local minimizer is guaranteed. In fact, there will always be a local minimizer that has rational entries with polynomial bitsize~\cite{vavasis1990quadratic}.

\begin{theorem}\label{Thm: Finding LM NP-hard}
	If there is a polynomial-time algorithm that finds a point within Euclidean distance $c^n$ (for any constant $c \ge 0$) of a local minimizer of an $n$-variate quadratic function over a polytope, then $P = NP$.
\end{theorem}

\begin{proof}
	Fix any constant $c \ge 0$. We show that if an algorithm could take as input a quadratic program with a bounded feasible set and in polynomial time return a point within distance $c^n$ of any local minimizer, then this algorithm would solve the STABLESET problem in polynomial time.
	
	Let a graph $G$ on $n$ vertices with adjacency matrix $A$ and a positive integer $r \le n$ be given. Let $k = r - 0.5$, $q_{A,k}$ be the quadratic form defined in (\ref{Def: qAk}), and consider the quadratic program
	\begin{equation}\label{Eq: SS Approx QP}
		\begin{aligned}
			& \underset{x \in \Rn}{\min}
			& & q_{A,k}(x) \\
			& \text{subject to}
			&& x \ge 0,\\
			&&& \sum_{i=1}^n x_i \le \ceil{3 c^n \sqrt{n}}.
		\end{aligned}
	\end{equation}
	Note that the feasible set of this problem is bounded. Moreover, the number of bits required to write down this quadratic program is polynomial in $n$. Indeed, the scalar $\ceil{3c^n\sqrt{n}}$ takes at most $2+n\lceil \log_2(c+1) \rceil+\frac{1}{2}\lceil \log_2(n+1) \rceil$ bits to write down, and the remaining $O(n^2)$ numbers in the problem data are bounded in magnitude by $n$, so they each take $O(\log_2(n))$ bits to write down.
	
	We will show that if $\alpha(G) < k$, the origin is the unique local minimizer of (\ref{Eq: SS Approx QP}), and that if $\alpha(G) \ge k$ (or equivalently $\alpha(G) > k$), any local minimizer $\xbar$ of (\ref{Eq: SS Approx QP}) satisfies $\sum_{i=1}^n \xbar_i = \ceil{3 c^n \sqrt{n}}$. Since the (Euclidean) distance from the origin to the hyperplane $\{ x \in \Rn | \sum_{i=1}^n x_i = \ceil{3 c^n \sqrt{n}}\}$ is at least $3c^n$, there is no point that is within distance $c^n$ of both the origin and this hyperplane. Thus, the graph $G$ has no stable set of size $r$ (or equivalently $\alpha(G) < k$) if and only if the Euclidean norm of all points within distance $c^n$ of any local minimizer of (\ref{Eq: SS Approx QP}) is less than or equal to $c^n$.
	
	To see why $\alpha(G) < k$ implies that the origin is the unique local minimizer of (\ref{Eq: SS Approx QP}), recall from the second claim of Corollary~\ref{Cor: Motzkin Straus 2} that the quartic form $p_{A,k}$ defined in (\ref{Def: pAk}) must be positive definite. Thus, for any nonzero vector $x \ge 0$, we have $q_{A,k}(x) > 0$. This implies that the origin is a local minimizer of (\ref{Eq: SS Approx QP}). Moreover, since $q_{A,k}$ is homogeneous, we have that no other feasible point can be a local minimizer. Indeed, for any nonzero vector $x \ge 0$ and any nonnegative scalar $\epsilon < 1$, $q_{A,k}(\epsilon x) < q_{A,k}(x)$.
	
	To see why when $\alpha(G) > k$, the last constraint of (\ref{Eq: SS Approx QP}) must be tight at all local minimizers, recall from the proof of Theorem~\ref{Thm: LM QP NP hard} that when $\alpha(G) > k$, the optimization problem in (\ref{Eq: SS QP}) has no local minimizer. Therefore, for any vector $x$ that is feasible to (\ref{Eq: SS Approx QP}) and satisfies $\sum_{i=1}^n x_i < \ceil{3 c^n \sqrt{n}}$, there exists a sequence $\{y_i\} \subseteq \Rn$ with $y_i \to x$, and satisfying $$y_i \ge 0, \sum_{i=1}^n y_i < \ceil{3 c^n \sqrt{n}}, q_{A,k}(y_i) < q_{A,k}(x), \forall i.$$ As the points $y_i$ are feasible to (\ref{Eq: SS Approx QP}), any vector $x$ satisfying $\sum_{i=1}^n x_i < \ceil{3 c^n \sqrt{n}}$ cannot be a local minimizer of (\ref{Eq: SS Approx QP}). Thus, if $\alpha(G) > k$, any local minimizer $\xbar$ of (\ref{Eq: SS Approx QP}) satisfies $\sum_{i=1}^n \xbar_i = \ceil{3 c^n \sqrt{n}}$.
	
\end{proof}

By replacing the quantity $\ceil{3 c^n \sqrt{n}}$ in the proof of Theorem~\ref{Thm: Finding LM NP-hard} with $\ceil{3n^{c+0.5}}$ and $2n$ respectively, we get the following two corollaries.

\begin{cor}\label{Cor: QP Pseudo}
	If there is a pseudo-polynomial time algorithm that finds a point within Euclidean distance $n^c$ (for any constant $c \ge 0$) of a local minimizer of an $n$-variate quadratic function over a polytope, then $P = NP$.
\end{cor}
\begin{cor}\label{Cor: 2n NP-hard}
	If there is a polynomial-time algorithm that finds a point within Euclidean distance $\epsilon \sqrt{n}$ (for any constant $\epsilon \in [0,1)$) of a local minimizer of a restricted set of quadratic programs over $n$ variables whose numerical data are integers bounded in magnitude by $2n$, then $P = NP$.
\end{cor}

In~\cite{pardalos1992open}, Pardalos and Vavasis also propose two follow-up questions about quadratic programs with compact feasible sets. The first is about the complexity of finding a ``KKT point''. As is, our proof does not have any implications for this question since the origin is always a KKT point of the quadratic programs that arise from our reductions. The second question asks whether finding a local minimizer is easier in the special case where the problem only has one local minimizer (which is thus also the global minimizer). Related to this question, we can prove the following claim.

\begin{theorem}\label{Thm: UniqueQP NP Hard}
	If there is a polynomial-time algorithm which decides whether a quadratic program with a bounded feasible set has a unique local minimizer, and if so returns this minimizer\footnote{This unique local (and therefore global) minimizer is guaranteed to have rational entries with polynomial bitsize; see~\cite{vavasis1990quadratic}.}, then P=NP.
\end{theorem}

\begin{proof}
	Suppose there was such an algorithm (call it Algorithm U). We show that Algorithm U would solve the STABLESET problem in polynomial time. Let a graph $G$ on $n$ vertices with adjacency matrix $A$ and a positive integer $r \le n$ be given, and input the quadratic program (\ref{Eq: SS Approx QP}), with $k = r-0.5$, into Algorithm U. Observe from the proof of Theorem~\ref{Thm: Finding LM NP-hard} that there are three possibilities for this quadratic program: (i) the origin is the unique local minimizer, (ii) there is a unique local minimizer and it is on the hyperplane $\{ x \in \Rn | \sum_{i=1}^n x_i = \ceil{3 c^n \sqrt{n}}\}$, and (iii) there are multiple local minimizers and they are all on the hyperplane $\{ x \in \Rn | \sum_{i=1}^n x_i = \ceil{3 c^n \sqrt{n}}\}$. Case (i) indicates that $\alpha(G) < k,$ and the output of Algorithm U in this case would be the origin. Cases (ii) and (iii) both indicate that $\alpha(G) > k$. The output of Algorithm U is a point away from the origin in case (ii), and the declaration that the local minimizer is not unique in case (iii). Thus Algorithm U would reveal which case we are in, and that would allow us to decide if $G$ has a stable set of size $r$ in polynomial time.
\end{proof}

To conclude, we have established intractability of several problems related to local minima of quadratic programs. We hope our results motivate more research on identifying classes of quadratic programs where local minimizers can be found more efficiently than global minimizers. One interesting example is the case of the concave knapsack problem, where Mor{\'e} and Vavasis~\cite{more1990solution} show that a local minimizer can be found in polynomial time even though unless P=NP, a global minimizer cannot.

\bibliographystyle{abbrv}
\bibliography{QP_refs}

\end{document}